\documentclass[a4paper,11pt,draft]{article}
\usepackage{amsmath,amssymb,enumerate,color}
\usepackage{amsthm,color}
\usepackage{txfonts}
\newcommand{\R}{\mathbb{R}}

\newcommand{\N}{\mathbb{N}}
\newcommand{\ep}{\varepsilon}
\newcommand{\pa}{\partial}

\DeclareMathOperator{\lifespan}{LifeSpan}

\newtheorem{theorem}{Theorem}[section]
\newtheorem{lemma}[theorem]{Lemma}
\newtheorem{proposition}[theorem]{Proposition}

\theoremstyle{remark}
\newtheorem{remark}{Remark}[section]
\theoremstyle{definition}

\newtheorem{definition}{Definition}[section]

\numberwithin{equation}{section}

\makeatletter
\def\@cite#1#2{[{{\bfseries #1}\if@tempswa , #2\fi}]}
\makeatother                  
\begin{document}
\begin{center}
\Large{{\bf
Life-span of solutions to semilinear wave equation 
\\with time-dependent critical damping 
\\
for specially localized initial data
}}
\end{center}

\vspace{5pt}

\begin{center}
Masahiro Ikeda%
\footnote{
Department of Mathematics, Faculty of Science and Technology, Keio University, 3-14-1 Hiyoshi, Kohoku-ku, Yokohama, 223-8522, Japan/Center for Advanced Intelligence Project, RIKEN, Japan, 
E-mail:\ {\tt masahiro.ikeda@keio.jp/masahiro.ikeda@riken.jp}}
and
Motohiro Sobajima%
\footnote{
Department of Mathematics, 
Faculty of Science and Technology, Tokyo University of Science,  
2641 Yamazaki, Noda-shi, Chiba, 278-8510, Japan,  
E-mail:\ {\tt msobajima1984@gmail.com}}
\end{center}

%%%%     ABSTRACT     %%%%%%
\newenvironment{summary}{\vspace{.5\baselineskip}\begin{list}{}{%
     \setlength{\baselineskip}{0.85\baselineskip}
     \setlength{\topsep}{0pt}
     \setlength{\leftmargin}{12mm}
     \setlength{\rightmargin}{12mm}
     \setlength{\listparindent}{0mm}
     \setlength{\itemindent}{\listparindent}
     \setlength{\parsep}{0pt}
     \item\relax}}{\end{list}\vspace{.5\baselineskip}}
\begin{summary}
{\footnotesize {\bf Abstract.}
This paper is concerned with the blowup phenomena for 
initial value problem of semilinear wave equation 
with critical time-dependent damping term
\begin{equation}\label{DW}
\tag*{(DW)}
\begin{cases}
\pa_t^2 u(x,t) -\Delta u(x,t) + \dfrac{\mu}{1+t}\pa_t u(x,t)=|u(x,t)|^p, 
& (x,t)\in \R^N \times (0,T),
\\
u(x,0)=\ep f(x), 
& x\in \R^N,
\\ 
\pa_t u(x,0)=\ep g(x),  
& x\in \R^N,
\end{cases}
\end{equation}
where $N\in \N$, $\mu\in [0,\frac{N^2+N+2}{N+2})$ and $\ep>0$ is a parameter 
describing the smallness of initial data. Given data $f,g$ are compactly supported in a small area inside of the unit ball. 
The result is the sharp upper bound of 
lifespan of solution $u$ with respect to the small parameter $\ep$ 
when $p_F(N)\leq p\leq p_0(N+\mu)$, where 
$p_F(N)$ denotes the Fujita exponent for 
the nonlinear heat equations
and $p_0(n)$ denotes 
the Strauss exponent for nonlinear wave equation \ref{DW} in $n$-dimension with $\mu=0$. 
Consequently, by connecting 
the result of D'Abbicco--Lucente--Reissig 
\cite{DLR15}, 
our result clarifies 
the threshold exponent $p_0(N+\mu)$
for dividing blowup phenomena and global existence of small solutions 
when $N=3$.  
The crucial idea is to construct suitable 
test functions satisfying the conjugate linear equation 
$\pa_t^2\Phi-\Delta \Phi-\pa_t(\frac{\mu}{1+t}\Phi)=0$
of \ref{DW} including the Gauss hypergeometric functions; 
note that  
the construction of test functions 
is different from Zhou--Han \cite{ZH14}.
}
\end{summary}



%%%%%%%%%%%%%%%%%%%%%%%%%%%%%%%%%%%%%%%%%%%%%%%%%%%%%%%%%%%
%%%%%%                                               %%%%%%
%                                                         %
%           Section 1 : Introduction and result           %
%                                                         %
%%%%%%                                               %%%%%%
%%%%%%%%%%%%%%%%%%%%%%%%%%%%%%%%%%%%%%%%%%%%%%%%%%%%%%%%%%%
\section{Introduction}
In this paper we consider the blowup phenomena for 
initial value problem of semilinear 
wave equation with scale-invariant damping term of time-dependent type as follows:
\begin{equation}\label{ndw}
\begin{cases}
\pa_t^2 u(x,t) -\Delta u(x,t) + b(t)\pa_t u(x,t)=|u(x,t)|^p, 
& (x,t)\in \R^N \times (0,T),
\\
u(x,0)=\ep f(x), 
& x\in \R^N,
\\ 
\pa_t u(x,0)=\ep g(x),  
& x\in \R^N. 
\end{cases}
\end{equation}
where $N\in\N$, $b(t)=\mu (1+t)^{-1}$ $(\mu\geq 0)$, 
$\ep>0$ is a small parameter
and 
$f,g$ 
are smooth nonnegative functions satisfying $f+g\not\equiv 0$
with a specially localized support
\begin{align}
\label{ini.supp}
{\rm supp}(f,g)\subset \overline{B(0,r_0)}=\{x\in\R^N\;;\;|x|\leq r_0\}
\end{align}
for some $r_0<1$. 

The purpose of the present paper is 
to prove the blowup phenomena for solutions of \eqref{ndw} 
with initial data satisfying (\ref{ini.supp})
and to give the sharp upper bound for lifespan of them.

We first recall the local well-posedness of \eqref{ndw}.
The following 
is well-known (e.g, Cazenave--Haraux \cite{CHbook}).
\begin{proposition}
\label{prop:exi-uni}
Let $1\le p<\infty$ for $N=1,2$ and 
$1\le p<\frac{N}{N-2}$ for $N\geq 3$. 
Then for every $(f,g)\in H^1(\R^N)\cap L^2(\R^N)$, 
there exist $T>0$ and a unique solution $u$ of \eqref{ndw} 
on $[0,T]$ in the following sense: 
\[
u\in X_T=
C([0,T];H^1(\R^N))
\cap
C^1([0,T];L^2(\R^N))
\cap 
C^2([0,T];H^{-1}(\R^N))
\]
and $\pa_t^2u-\Delta u+b(t)\pa_t u=|u|^p$ in $H^{-1}(\R^N)$. 
\end{proposition}

The study of blowup phenomena of solutions to \eqref{ndw} 
with $\mu=0$ has been studied by many mathematicians 
(e.g., \cite{John79,Kato80,Strauss81,YZ06,Zhou07} and their reference therein). 
In particular, 
the threshold exponent 
for dividing blowup phenomena and global existence of small solutions 
is clarified as $p_0(N)$ (for example, see 
Yordanov--Zhang \cite{YZ06} and Zhou \cite{Zhou07}), 
where $p_0(n)$ is called Strauss exponent and 
is given as the positive root of the quadratic equation 
\[
(n-1)p^2-(n+1)p-2=0.
\]
The sharp upper estimate for lifespan of \eqref{ndw} with $\mu=0$ 
is given in Takamura--Wakasa \cite{TW11}, where 
the precise definition of the lifespan of solutions to \eqref{ndw} as follows: 
\begin{definition}\label{def:lifespan}
We denote $\lifespan(u)$ as the maximal existence time for solution of \eqref{ndw}, that is, 
\[
\lifespan(u)=\sup\{T>0\;;\;u\in X_T\ \& \ \text{$u$ is a solution of \eqref{ndw} in $(0,T)$}\},
\]
where 
$X_T=C([0,T];H^1(\R^N))
\cap
C^1([0,T];L^2(\R^N))
\cap 
C^2([0,T];H^{-1}(\R^N))
$, as in Proposition \ref{prop:exi-uni}. 
\end{definition}
Later, the alternative approach for deriving 
the estimate
in the critical case $p=p_0(N)$ appears in Zhou--Han \cite{ZH14}.
Their approach could be generalized for dealing with 
the blowup phenomena of 
the singular problem
\begin{equation}\label{ndw2}
\begin{cases}
\pa_t^2 u(x,t) -\Delta u(x,t) + \dfrac{V_0}{|x|}\pa_t u(x,t)=|u(x,t)|^p, 
& (x,t)\in \R^N \times (0,T),
\\
u(x,0)=\ep f(x), 
& x\in \R^N,
\\ 
\pa_t u(x,0)=\ep g(x),  
& x\in \R^N
\end{cases}
\end{equation}
with $N\geq 3$ and $V_0\in [0,\frac{(N-1)^2}{N+1})$ (see \cite{IkedaSobajima1})
(the technique of test function by Zhou--Han was simplified in \cite{IkedaSobajima1}). 

After that the study for the case $\mu>0$ has been 
considered in Wakasugi \cite{Wakasugi14, Wakasugi-PhD}, 
D'Abbicco \cite{DAbbicco15}, 
D'Abbicco--Lecente--Reissig \cite{DLR15},
Wakasa \cite{Wakasa16}
and Lai--Takamura--Wakasa \cite{LTWarxiv} and others. 
In particular, the blowup phenomena 
with certain upper bounds of the lifespan
are proved in the following cases: 
\begin{align*}
\begin{cases}
1<p\leq p_F(N)&\text{for\ }N\in\N, \quad\mu\geq 1,
\\
1<p\leq 1+\frac{2}{N-1+\mu}&\text{for\ }N\in\N,\quad 0<\mu\leq 1,
\end{cases}
\end{align*}
in \cite{Wakasugi14, Wakasugi-PhD}, and 
\begin{align*}
\begin{cases}
1<p\leq p_F(N)&\text{for\ }N=1, \quad \mu=2,
\\
1<p<p_0(N+2)&\text{for\ }N=2,3, \quad\mu=2,
\end{cases}
\end{align*}
in \cite{DLR15} and \cite{Wakasa16}. Here 
$p_F(N)$ denotes the Fujita exponent $1+\frac{2}{N}$. 
Recently, it is proved in \cite{LTWarxiv} that 
\begin{align*}
p_F(N)<p\leq p_0(N+2\mu)\quad \text{for\ }N\geq 2, \quad 0<\mu<\frac{N^2+N+2}{2(N+2)}. 
\end{align*}
On the one hand, global existence of small solutions (SDGE) 
is obtained for the following cases in 
\cite{DAbbicco15}:
\begin{align*}
\begin{cases}
p_F(1)<p<\infty &\text{for\ }N=1, \quad\mu\geq \frac{5}{3},
\\
p_F(2)<p<\infty &\text{for\ }N=2, \quad\mu\geq 3,
\\
p_F(N)<p\leq \frac{N}{N-2}&\text{for\ }N\geq 3, \quad\mu\geq N+2.
\end{cases}
\end{align*}
Moreover, in \cite{DLR15}
\begin{align*}
\begin{cases}
p_F(2)=p_0(4)=2<p &\text{for\ }N=2, \quad\mu=2,
\\
p_0(5)<p &\text{for\ }N=3, \quad\mu=2. 
\end{cases}
\end{align*}
We remark that if $\mu=2$ and $N=2,3$, 
then threshold for dividing the blowup phenomena and 
global existence of small solutions is clarified as $p_0(N+2)$. 
However, such a threshold for the other cases cannot be obtained so far. 

The purpose of the present paper is to give the explicit 
value for the threshold for blowup phenomena 
and to give the sharp upper bound of lifespan of solutions to 
\eqref{ndw}. 

To state the main result, we introduce the following notation.
\begin{definition}\label{def:strauss.exp}
We introduce the following quadratic polynomial
\[
\gamma(n;p)=2+(n+1)p-(n-1)p^2
\]
and denote $p_0(n)$ for $n>1$ as the positive root of 
the quadratic equation $\gamma(n;p)=0$ and additionally set $p_0(1)=\infty$. 
We also put 
\[
\mu_*=\frac{N^2+N+2}{N+2}.
\]
It is worth noticing that 
if $\mu\in [0,\mu_*)$, then $p_F(N)<p_0(N+\mu)$, and $p_F(N)=p_0(N+\mu_*)$. 
\end{definition}

Here we are in a position to state our main result of the present paper. 

\begin{theorem}\label{main}
Assume $p_F<p\leq p_0(N+\mu)$ with 
\begin{align*}
\begin{cases}
0< \mu <\frac{4}{3}
&\text{if }N=1,
\\[10pt]
0\leq  \mu <\mu_* %\frac{N^2+N+2}{N+2}
&\text{if }N\geq 2.
\end{cases}
\end{align*} 
Then there exists a constant $\ep_0>0$ such that 
for every $\ep\in (0,\ep_0)$, 
the solution $u_\ep$ of \eqref{ndw} satisfies 
$\lifespan(u_\ep)<\infty$. 
Moreover, if $N=1$, then 
for every $\ep\in (0,\ep_0)$, 
\begin{align*}
\lifespan(u_\ep)\leq 
\begin{cases}
\exp[C\ep^{-p(p-1)}]
&
\text{if } p=p_0(1+\mu),
\\
C_\delta\ep^{-\frac{2p(p-1)}{\gamma(N+\mu;p)}-\delta}
&\text{if}\ \max\{3,\frac{2}{\mu}\}\leq p<p_0(1+\mu),
\\
C'_\delta\ep^{-\frac{2(p-1)}{\mu}-\delta} 
& \text{if }0<\mu<\frac{2}{3}, 3\leq p<\frac{2}{\mu}
\end{cases}
\end{align*}
and if $N\geq 2$, then 
for every $\ep\in (0,\ep_0)$, 
\begin{align*}
\lifespan(u_\ep)\leq 
\begin{cases}
\exp[C\ep^{-p(p-1)}]
&
\text{if } p=p_0(N+\mu),
\\
C_\delta\ep^{-\frac{2p(p-1)}{\gamma(N+\mu;p)}-\delta}
&\text{if}\ p_0(N+2+\mu)\leq p<p_0(N+\mu),
\\
C'_\delta\ep^{-1-\delta} 
& \text{if }p<p_0(N+2+\mu).
\end{cases}
\end{align*}
Here $C,C_\delta,C'_\delta$ are constants independent of $\ep>0$ 
and $\delta>0$ can be chosen arbitrary small. 
\end{theorem}

%\begin{remark}
%In the case $p<p_0(N+\mu)$ we can take $\ep_0>0$ arbitrary large. 
%\end{remark}

We would like to show the relation between several previous works 
by the following three figures.

%%%%%%%%%%%%%%%%%%%%%%%%%%%%%%%%%%%%%%%%%%%%%%%%%%%%%%%%%%%
%%%%%%%%%%%%%%%%%%%%%%%%%%%%%%%%%%%%%%%%%%%%%%%%%%%%%%%%%%%
%%%%%%%%%%%%%%%%%%%%%%%%%%%%%%%%%%%%%%%%%%%%%%%%%%%%%%%%%%%
%%%%%%%%%%%%%%%%%%%%%%%%%%%%%%%%%%%%%%%%%%%%%%%%%%%%%%%%%%%
%%%%%%%%%%%%%%%%%%%%%%%%%%%%%%%%%%%%%%%%%%%%%%%%%%%%%%%%%%%

\bigskip 

\begin{center}
\input{1.fig}

\bigskip 

\input{2.fig}

\bigskip 

\input{3.fig}

\bigskip 

\end{center}

%%%%%%%%%%%%%%%%%%%%%%%%%%%%%%%%%%%%%%%%%%%%%%%%%%%%%%%%%%%
%%%%%%%%%%%%%%%%%%%%%%%%%%%%%%%%%%%%%%%%%%%%%%%%%%%%%%%%%%%
%%%%%%%%%%%%%%%%%%%%%%%%%%%%%%%%%%%%%%%%%%%%%%%%%%%%%%%%%%%
%%%%%%%%%%%%%%%%%%%%%%%%%%%%%%%%%%%%%%%%%%%%%%%%%%%%%%%%%%%
%%%%%%%%%%%%%%%%%%%%%%%%%%%%%%%%%%%%%%%%%%%%%%%%%%%%%%%%%%%
%%%%%%%%%%%%%%%%%%%%%%%%%%%%%%%%%%%%%%%%%%%%%%%%%%%%%%%%%%%

\begin{remark}
In the one dimensional case, Kato proved in \cite{Kato80} 
that for every $1<p<\infty$, the solution of \eqref{ndw} with $\mu=0$ 
blowup at finite time. Therefore the Strauss exponent
can be understood as $p_0(1)=\infty$. On the other hand, 
Theorem \ref{main} means that 
if $\mu>0$, then the Strauss exponent is finite 
and one can expect that $p_0(1+\mu)$ is the threshold for $p$
to divide the blowup phenomena and global existence of small solutions. 
\end{remark}

\begin{remark}
If $N\geq 3$, then we can take $\mu=2$ in Theorem \ref{main}.
In this case, the threshold can be $p_0(N+2)$ 
which is already obtained by D'Abbicco--Lecente--Reissig \cite{DLR15} 
only when $N=3$.
Theorem \ref{main} extends it to the general $\mu\in [0,\mu_*)$ 
and the result connects the threshold number continuously from $\mu=0$ to $\mu>0$. 
\end{remark}

\begin{remark}
Lai--Takamura--Wakasa proved in \cite{LTWarxiv} that 
for $p_F(N)\leq p< p_0(N+2\mu)$ with $\mu\in [0,\frac{\mu_*}{2})$, 
the lifespan of solution $u_\ep$ is estimated above by 
$
\ep^{-\frac{2p(p-1)}{\gamma(N+2\mu;p)}}.
$
Therefore Theorem \ref{main} gives a wider range of $(p,\mu)$
for blowup phenomena 
and better estimates of lifespan $\ep^{-\frac{2p(p-1)}{\gamma(N+\mu;p)}-\delta}$
than that in \cite{LTWarxiv}. 
Moreover, since we could prove 
the almost global existence of solutions when $p=p_0(N+\mu)$, 
we may find the threshold $p_0(N+\mu)$
for dividing the blowup phenomena and global existence of small solutions 
and this is true when $N=3$ and $\mu=2$. 
\end{remark}

The present paper is organized as follows. 
The special solution of conjugate linear equation of \eqref{ndw} is 
constructed and analysed in Section 2. 
In Section 3, 
by using the the function constructed in Section 2 
as a test function, 
we prove Theorem \ref{main} 
for subcritical case and critical case separately. 
Although the proof is similar to the one 
in \cite{IkedaSobajima1}, 
we give a complete proof for self-containedness.  

\section{Special solutions of linear damped wave equation}

In this section we construct special solutions of linear damped wave equation. 

Here we start with properties for  self-similar solutions of \eqref{ndw}. 
First we set 
\[
\mathcal{Q}_0=\{(x,t)\in \R^{N+1}\;;\;|x|<t\}.
\]

\begin{lemma}\label{lem:self-sim}
Let  $\Psi\in C^2(\mathcal{Q}_0)$ be radially symmetric. 
Assume that $\Psi$ satisfies $\pa_t^2\Psi-\Delta \Psi+\frac{\mu'}{t}
\pa_t\Psi=0$ and 
\[
\Psi(x,t)=\lambda^{\beta}\Psi(\lambda x,\lambda t), \quad (x,t)\in \mathcal{Q}_0
\]
for all $\lambda>0$. Then $\Psi$ can be represented by 
\[
\Psi(x,t)=t^{-\beta}\psi\left(\frac{|x|^2}{t^2}\right)
\]
with $\psi\in C^2([0,1))$ satisfying the Gauss hypergeometric differential equation
\begin{equation}\label{Gauss}
z(1-z)\psi''(z)+\left(\frac{N}{2}-\left(1+\frac{\beta}{2}+\frac{\beta+1-\mu'}{2}\right)z\right)\psi'(z)-\frac{\beta(\beta+1-\mu')}{4}\psi(z)=0, \quad z\in (0,1).
\end{equation}
\end{lemma}

\begin{proof}
By assumption we can take $\lambda=t^{-1}$ and then 
we have 
\[
\Psi(x,t)=t^{-\beta}\psi\left(\frac{|x|^2}{t^2}\right),
\]
where we have put $\psi(z)=\Psi(\sqrt{z},1)$. Noting that 
\begin{align*}
\pa_t \Psi(x,t)
&=
t^{-\beta-1}
\left(-\beta \psi\left(\frac{|x|^2}{t^2}\right)- \frac{2|x|^2}{t^2}\psi'\left(\frac{|x|^2}{t^2}\right)\right)
=
t^{-\beta-1}(-\beta \psi(z)- 2z\psi'(z))
\end{align*}
and 
\[
\pa_t^2 \Psi(x,t)
=
t^{-\beta-2}(\beta(\beta+1)\psi(z)+ (4\beta+6)z\psi'(z)+4z^2\psi''(z)).
\]
Moreover, 
\begin{align*}
\Delta \Psi(x,t)
=t^{-\beta}
\left(
\frac{2N}{t^2}\psi'\left(\frac{|x|^2}{t^2}\right)
+
\frac{4|x|^2}{t^4}\psi''\left(\frac{|x|^2}{t^2}\right)
\right)
=t^{-\beta-2}
\left(
2N\psi'(z)
+
4z\psi''(z)
\right).
\end{align*}
Combining the above identities, we see that
\begin{align*}
t^{\beta+2}
\left(\pa_t^2\Psi-\Delta \Psi+\frac{\mu'}{t}\pa_t\Psi\right)
&=
4z(z-1)\psi''(z)
+
\left(-2N+(4\beta+6-2\mu')z\right)\psi'(z)
+ 
\beta(\beta+1-\mu')\psi(z) 
\\
&=
-4\left[
z(1-z)\psi''(z)
+
\left(\frac{N}{2}+\left(\beta+\frac{3}{2}-\frac{\mu'}{2}\right)z\right)\psi'(z)
- 
\frac{\beta(\beta+1-\mu')}{4}\psi(z) 
\right].
\end{align*}
This implies that $\psi$ satisfies \eqref{Gauss}.
\end{proof}

\begin{definition}\label{def:phi}
For $\beta>0$ and $\mu\in\R$, set 
\[
\Psi_{\beta,\mu}(x,t)
=
(1+t)^{-\beta}
\psi_{\beta,\mu}\left(\frac{|x|^2}{(1+t)^2}\right), 
\quad 
\psi_{\beta,\mu}(z)=F\left(\frac{\beta}{2},\frac{\beta-1+\mu}{2},\frac{N}{2};z\right),
\]
where $F(a,b,c;z)$ is the Gauss hypergeometric function given by 
\begin{align*}
F(a,b,c;z)=\sum_{n=0}^\infty\frac{(a)_n(b)_n}{(c)_n}\,\frac{z^n}{n!}
\end{align*}
with $(d)_0=1$ and $(d)_n=\prod_{k=1}^n (d+k-1)$ for $n\in\N$.
\end{definition}

Here some properties of $\Psi_{\beta,\mu}$ is collected 
in the following lemma, which are due to well-known fact for 
the Gauss hypergeometric functions
(see e.g., Beals--Wong \cite{BW}). 
\begin{lemma}\label{lem:psi-prop}
{\bf (i)} For every $\beta>0$ and $\mu\in\R$, 
\begin{align*}
\pa_t^2\Psi_{\beta,\mu}(x,t)
-\Delta \Psi_{\beta,\mu}(x,t)
+\frac{2-\mu}{1+t}\pa_t\Psi_{\beta,\mu}(x,t)=0, \quad (x,t)\in \mathcal{Q}_1.
\end{align*}
\noindent
{\bf (ii)} 
If $\max\{0,1-\mu\} <\beta< \frac{N+1-\mu}{2}$,  
then there exists a constant $c_{\beta,\mu}>1$ such that 
\[
1 \leq \psi_{\beta,\mu}(z)\leq c_{\beta,\mu}.
\] 
\noindent
{\bf (iii)} 
If $\beta > \frac{N+1-\mu}{2}$,
then there exists a constant $c'_{\beta,\mu}>1$ such that 
\[
(c'_{\beta,\mu})^{-1}
\left(1-
\sqrt{z}\right)^{\frac{N+1-\mu}{2}-\beta}
\leq \psi_{\beta,\mu}(z)\leq 
c_{\beta,\mu} 
\left(1-
\sqrt{z}\right)^{\frac{N+1-\mu}{2}-\beta}.
\] 
\noindent
{\bf (iv)} 
For every $\beta>0$ and $\mu\in\R$, 
$\psi_{\beta+2,\mu-2}(z)\geq \psi_{\beta,\mu}(z)$ and 
\[
\pa_t\Psi_{\beta,\mu}(x,t)
=
-\beta (1+t)\Psi_{\beta+2,\mu-2}(x,t).
\]
\end{lemma}

\begin{proof}
The assertion {\bf (i)} is a direct consequence of Lemma \ref{lem:self-sim}. 
Moreover, we see from the behavior of $F(a,b,c;z)$ that 
{\bf (ii)} and {\bf (iii)} hold.
For {\bf (iv)}, the first assertion directly follows from the definition 
of $F(a,b,c;z)$. Indeed, since $F$ is monotone increasing with respect to 
$a$ and then we have 
\[
\psi_{\beta,\mu}(z)=F\left(\frac{\beta}{2},\frac{\beta-1+\mu}{2},\frac{N}{2};z\right)
\leq 
F\left(\frac{\beta+2}{2},\frac{\beta-1+\mu}{2},\frac{N}{2};z\right)
\leq 
\psi_{\beta+2,\mu-2}(z).
\]
We prove the second assertion $\pa_t\Psi_{\beta,\mu}=-\beta (1+t)\Psi_{\beta+2,\mu-2}$.
Since it follows from the proof of Lemma \ref{lem:self-sim} that 
\[
(1+t)^{-1}
\pa_t\Psi_{\beta,\mu}(x,t)=
-(1+t)^{-\beta-2}(\beta \psi_{\beta,\mu}(z) + 2z\psi_{\beta,\mu}'(z))
\]
with $z=|x|^2/(1+t)^2$. Therefore it suffices to show that 
\begin{equation}\label{gauss.aux}
\beta \psi_{\beta,\mu}(z) + 2z\psi_{\beta,\mu}'(z)=
\beta \psi_{\beta+2,\mu-2}(z),
\quad 
 z\in (0,1).
\end{equation}
Put $\psi(z)=\beta \psi_{\beta,\mu}(z) + 2z\psi_{\beta,\mu}'(z)$ for $z\in [0,1)$. 
Then by the definition of $F(\cdot,\cdot,\cdot;z)$, 
we have $\psi(0)=\beta$. On the other hand, we see from 
the Gauss hypergeometric equation with 
$a=\frac{\beta}{2}$, $b=\frac{\beta-1+\mu}{2}$ and $c=\frac{N}{2}$
that 
\begin{align*}
(1-z)\psi'(z)
&=(1-z)\Big((\beta+2) \psi_{\beta,\mu}'(z)
 + 2z\psi_{\beta,\mu}''(z)\Big)
\\
&=(\beta+2)(1-z)\psi_{\beta,\mu}'(z) + 2z(1-z)\psi_{\beta,\mu}''(z)
\\
&=2(a+1)(1-z)\psi_{\beta,\mu}'(z) 
- 2(c-(1+a+b)z)\psi_{\beta,\mu}'(z)+2ab\psi_{\beta,\mu}(z)
\\
&=2(a+1-c)\psi_{\beta,\mu}'(z) 
+2bz\psi_{\beta,\mu}'(z)+2ab\psi_{\beta,\mu}(z)
\\
&=2(a+1-c)\psi_{\beta,\mu}'(z) 
+b\psi(z)
\end{align*}
and therefore $(1-z)\psi'(z)-b\psi(z)=2(a+1-c)\psi_{\beta,\mu}'(z)$. 
The definition of $\psi$ yields 
\begin{align*}
z(1-z)\psi'(z)-bz\psi(z)
&=2(a+1-c)z\psi_{\beta,\mu}'(z)
\\
&=(a+1-c)\psi(z)-2(a+1-c)a\varphi(z)
\end{align*} 
Differentiating the above equality, we have
\begin{align*}
z(1-z)\psi''(z)+(1-(2+b)z)\psi'(z)-b\psi(z)
&=(a+1-c)\psi'(z)-2(a+1-c)a\psi_{\beta,\mu}'(z)
\\
&=(a+1-c)\psi'(z)-a\Big((1-z)\psi'(z)-b\psi(z)\Big).
\end{align*}
Hence we have $z(1-z)\psi''(z)+(c-(2+a+b)z)\psi'(z)+(a+1)b\psi(z)=0$. 
All solutions of this equation having a bounded derivative 
near $0$ can be written by $\psi(z)=h F(a+1,b,c;z)$ with $h\in \R$. 
Combining the initial value $\psi(0)=\beta$, we obtain {\bf (iv)}. 
\end{proof}

\begin{definition}\label{def:phi}
For $\mu>0$ and $\beta>\max\{0,1-\mu\}$, set 
\[
\Phi_{\beta}(x,t)
=
(1+t)\Psi_{\beta,\mu}(x,1+t). 
\]
\end{definition}

\begin{lemma}\label{lem:textfunctions}
For $\beta>0$, $\Phi_\beta$ satisfies the conjugate equation of \eqref{ndw}:
\begin{equation}\label{eq:dual}
\pa_t^2\Phi_{\beta}-\Delta \Phi_{\beta}-
\pa_t\left(\frac{\mu}{1+t}\Phi_{\beta}\right)=0, 
\quad \text{in}\ 
\mathcal{Q}_1=\{(x,t)\in\R^N\times (0,\infty)\;;\;|x|<1+t\}. 
\end{equation}
\end{lemma}

\begin{proof}
By direct calculation we see that 
\begin{align*}
\pa_t^2\Phi_{\beta}-\Delta \Phi_{\beta}-
\pa_t\left(\frac{\mu}{1+t}\Phi_{\beta}\right)
&=
(1+t)
\pa_t^2\Psi_{\beta,\mu}
+
2\pa_t\Psi_{\beta,\mu}
-(1+t)\Delta \Psi_{\beta,\mu}-
\mu \pa_t\Psi_{\beta,\mu}
\\
&=(1+t)\left[
\pa_t^2\Psi
-\Delta \Psi+\frac{2-\mu}{1+t}\pa_t\Psi
\right].
\end{align*}
Applying Lemma \ref{lem:psi-prop} {\bf (i)}, we obtain \eqref{eq:dual}.
\end{proof}

\section{Proof of Blowup phenomena and estimates for lifespan}

\subsection{Preliminaries for showing blowup phenomena}
We first state a criterion for derivation of 
upper bound for lifespan. 

The following assertion 
is essentially proved in \cite[Section 3]{ZH14}.
For the detail see e.g., 
\cite{IkedaSobajima1}.
\begin{lemma}\label{lem:blowup}
Let $H\in C^2([\sigma_0,\infty)$ be nonnegative function. 

\noindent{\bf (i)}
Assume that there exists positive constants $c,C,C'$ such that
\begin{align*}
C[H(\sigma)]^{p}\leq H''(\sigma)+c\,\frac{H'(\sigma)}{\sigma}
\end{align*}
with $H(\sigma)\geq \ep^pC\sigma^2$ and $H'(\sigma)\geq \ep^p C\sigma$. 
Then there exists a positive constant $\ep_0$ such that 
if $\ep\in (0,\ep_0)$, then 
$H$ blows up before $\sigma=C''\ep^{-\frac{p-1}{2}}$ for some $C''>0$. 

\noindent{\bf (ii)}
Assume that there exist positive constants $c,C,C'$ such that
\begin{align*}
C\sigma^{1-p}[H(\sigma)]^{p}\leq H''(\sigma)+2 H'(\sigma)
\end{align*}
with $H(\sigma)\geq \ep^pC\sigma$ and $H'(\sigma)\geq \ep^p C$. 
Then 
there exists a positive constant $\ep_0$ such that 
if $\ep\in (0,\ep_0)$, then 
$H$ blows up before $\sigma=C''\ep^{-p(p-1)}$ for some $C''>0$. 
\end{lemma}

We focus our eyes to the following functionals. 
\begin{definition}\label{def:G_beta}
For $\beta\in (\max\{0,1-\mu\},\frac{N+1-\mu}{2})$, define the following three functions
\begin{align*}
G_\beta(t)
&:=
\int_{\R^N}
  |u(x,t)|^{p}
  \Phi_{\beta}(x,t)
\,dx,\quad t\geq 0,
\\
H_\beta(t)
&:=
\int_0^t
  (t-s)(1+s)G_\beta(t)
\,ds, 
\quad t\geq 0,
\\
J_\beta(t)
&:=
\int_0^t
  (1+s)^{-3}H_\beta(t)
\,ds, 
\quad t\geq 0. 
\end{align*}
Note that we can see from Lemma \ref{lem:psi-prop} {\bf (ii)} that 
$G_\beta(t)\approx (1+t)^{1-\beta}\|u(t)\|_{L^p}^p$. 
\end{definition}

For the proof of blowup phenomena, 
we remark the following two lemmas. 

\begin{lemma}\label{lem:sufficient}
If $u_\ep$ is a solution of \eqref{ndw} in Proposition \ref{prop:exi-uni}
with parameter $\ep>0$, 
then $J_\beta$ does not blow up until $\lifespan (u_\ep)$. 
\end{lemma}

\begin{lemma}\label{lem:trick}For every $\beta>0$ and $t\geq 0$, 
\[
(1+t)^2J_\beta(t)=\frac{1}{2}\int_0^t(t-s)^2G_\beta(s)\,ds.
\]
\end{lemma}
%\begin{proof}
%This can be verified by integration by parts twice, by noting that 
%\[
%\frac{d}{ds}\Big((t-s)^2(1+s)^{-1}\Big)=\frac{2(1+t)^2}{(1+s)^3}. 
%\]
%\end{proof}

The following lemma is a unified approach 
of subcritical and critical case for blowup phenomena. 

\begin{lemma}\label{base}
Let $u$ be a solution of \eqref{ndw}. Then for every $\beta>0$ and $t\geq0$, 
\begin{align}
\nonumber
&
\ep E_{\beta,0}+
\ep E_{\beta,1}t+
\int_0^t(t-s)
G_\beta(s)\,ds
\\
\label{G-equality}
&= 
\int_{\R^N}u(t)\Phi_\beta(t)\,dx
+
\int_{0}^t
(1+s)^{-\beta}
\int_{\R^N}u(s)
\widetilde{\psi}_\beta\left(\frac{|x|^2}{(1+s)^2}\right)
\,dx\,ds,
\end{align}
where 
\begin{align*}
E_{\beta,0}
&=
\int_{\R^N}
   f(x)\psi_{\beta,\mu}(|x|^2)
\,dx>0.
\\
E_{\beta,1}&=
   \int_{\R^N}g(x)\psi_{\beta,\mu}(|x|^2)
\,dx
+
\int_{\R^N}f(x)
\Big(
   \beta\psi_{\beta+2,\mu-2}(|x|^2)
   + (\mu-1) \psi_{\beta,\mu}(|x|^2)
   \Big)\,dx
>0
\end{align*}
and 
\[
\widetilde{\psi}_\beta(z)
=2\beta
\psi_{\beta+2,\mu-2}(z)+(\mu-2)\psi_{\beta,\mu}(z).
\]
\end{lemma}

\begin{proof}
By the equation in \eqref{ndw} 
we see from integration by parts that
\begin{align*}
G_\beta(t)
&=
\int_{\R^N}
\left(
  \pa_t^2u(t)-\Delta u(t)+\frac{\mu}{1+t}\pa_tu(t)
\right)\Phi_\beta(t)
\,dx
\\
&=
\int_{\R^N}
\left(
  \pa_t^2u(t)+\frac{\mu}{1+t}\pa_tu(t)
\right)\Phi_\beta(t)
\,dx
-
\int_{\R^N}u(t)\Delta \Phi_\beta(t)\,dx.
\end{align*}
Using Lemma \ref{lem:textfunctions}, we have
\begin{align*}
G_\beta(t)&=
\int_{\R^N}
\left(
  \pa_t^2u(t)+\frac{\mu}{1+t}\pa_tu(t)
\right)\Phi_\beta(t)
\,dx
-
\int_{\R^N}u(t)
\left(\pa_t^2\Phi_\beta(t)-\pa_t\left(\frac{\mu}{1+t}\Phi_\beta(t)\right)\right)\,dx
\\
&=
\frac{d}{dt}
\left[
\int_{\R^N}\left(\pa_t u(t)\Phi_\beta(t)-u(t)\pa_t\Phi_\beta(t)\right)\,dx
+
\frac{\mu}{1+t}\int_{\R^N}u(t)\Phi_\beta(t)\,dx\right].
\end{align*}
Observing that 
\begin{align*}
&\int_{\R^N}\left(\pa_t u(0)\Phi_\beta(0)-u(0)\pa_t\Phi_\beta(0)\right)\,dx
+
\mu\int_{\R^N}u(0)\Phi_\beta(0)\,dx
\\
&=
\ep \int_{\R^N}\left(g(x)\Psi_{\beta,\mu}(x,0)
-f(x)(\Psi_{\beta,\mu}(x,0)-\beta\Psi_{\beta+2,\mu-2}(x,0))\right)\,dx
+
\ep \mu\int_{\R^N}f(x)\Psi_{\beta,\mu}(x,0)\,dx
\\
&=
\ep 
\left(
   \int_{\R^N}g(x)\psi_{\beta,\mu}(|x|^2)
\,dx
+
\int_{\R^N}f(x)
\Big(
   \beta\psi_{\beta+2,\mu-2}(|x|^2)
   + (\mu-1) \psi_{\beta,\mu}(|x|^2)
   \Big)\,dx
\right).
\end{align*}
It follows from $\beta-1+\mu>0$ that 
\begin{align*}
\beta\psi_{\beta+2,\mu-2}(z)+(\mu-1) \psi_{\beta,\mu}(z)
&\geq 
(\beta-1+\mu)\psi_{\beta,\mu}(z)
\end{align*}
and hence 
we have $E_{\beta,1}>0$. 
Integrating it over $[0,t]$, we have
\begin{align*}
&\ep E_{\beta,1}+
\int_0^t
G_\beta(s)\,ds
\\
&= 
\int_{\R^N}\left(\pa_t u(t)\Phi_\beta(t)-u(t)\pa_t\Phi_\beta(t)\right)\,dx
+
\frac{\mu}{1+t}\int_{\R^N}u(t)\Phi_\beta(t)\,dx
\\
&= 
\frac{d}{dt}\left[\int_{\R^N}u(t)\Phi_\beta(t)\,dx\right]
+
(1+t)^{-\beta}
\int_{\R^N}u(t)
\left(
2\beta
\psi_{\beta+2,\mu-2}\left(\frac{|x|^2}{(1+t)^2}\right)
+(\mu-2)\psi_{\beta,\mu}\left(\frac{|x|^2}{(1+t)^2}\right)
\right)
\,dx.
\end{align*}
Integrating it again, we obtain \eqref{G-equality}.
\end{proof}

\begin{lemma}\label{base2}
Assume $p_F(N) <p\leq \frac{N}{N-2}$ and 
\begin{align*}
\begin{cases}
0< \mu <\frac{4}{3}
&\text{if }N=1,
\\[10pt]
0\leq  \mu <\mu_* %\frac{N^2+N+2}{N+2}
&\text{if }N\geq 2. 
\end{cases}
\end{align*}
{\bf (i)}\ Let $q>p$ satisfy 
\[
\beta
=
\frac{N+1-\mu}{2}-\frac{1}{q}
>\max\{0,1-\mu\}.
\]
Then 
\begin{align*}
&\ep E_{\beta,0}+\ep E_{\beta,1} t+
\int_0^t
(t-s)G_\beta(s)\,ds
\leq
C_1
\left[
\|u(t)\|_{L^p}
(1+t)^{\frac{N}{p'}+1-\beta}
+
\int_0^t
\|u(s)\|_{L^p}
(1+s)^{\frac{N}{p'}+1-\frac{N+1-\mu}{2}-\frac{1}{p'}}
\,ds
\right].
\end{align*}
{\bf (ii)}\ If $N\geq 2$ or 
$N=1$ with $p>2\mu^{-1}$, then setting 
$
\beta_0=\frac{N+1-\mu}{2}-\frac{1}{p}>\max\{0,1-\mu\},
$
one has
\begin{align*}
&
\int_0^t
(t-s)G_\beta(s)\,ds
\leq
C_1
\left[
\|u(t)\|_{L^p}
(1+t)^{\frac{N}{p'}+1-\beta}
+
\int_0^t
\|u(s)\|_{L^p}
(1+s)^{\frac{N}{p'}-\beta_0}(\log (1+s))^{\frac{1}{p'}}
\,ds
\right].
\end{align*}

\end{lemma}
\begin{proof}
By Lemma \ref{base} with finite propagation property, we have 
\begin{align*}
&
\ep E_{\beta,0}+\ep E_{\beta,2} t+
\int_0^t
(t-s)G_\beta(s)\,ds
=I_{\beta,1}(t)+\beta I_{\beta,2}(t).
\end{align*}
where 
\begin{align*}
I_{\beta,1}(t)
&=
\int_{B(0,r_0+t)}u(x,t)\Phi_\beta(x,t)\,dx,
\\
I_{\beta,2}(t)
&=
\int_{0}^t(1+s)^{-\beta}\left(\int_{B(0,r_0+t)}u(x,s)
\widetilde{\psi}\left(\frac{|x|^2}{(1+s)^2}\right)\,dx\right)\,ds.
\end{align*}
Using Lemma \ref{lem:psi-prop} {\bf (ii)}, we have 
\begin{align*}
I_{\beta,1}(t)
&\leq 
\left(\int_{B(0,r_0+t)}|u(x,t)|^p\,dx\right)^\frac{1}{p}
\left(\int_{B(0,r_0+t)}\Phi_\beta(x,t)^{p'}\,dx\right)^\frac{1}{p'}
\\
&\leq 
c_{\beta,\mu}
N^{-\frac{1}{p'}}|S^{N-1}|^{\frac{1}{p'}}\|u(t)\|_{L^p}
(1+t)^{\frac{N}{p'}+1-\beta}.
\end{align*}
Noting that $\widetilde{\psi}(z)\leq (\beta+\mu)\psi_{\beta+2,\mu-2}(z)$, 
we see from Lemma \ref{lem:psi-prop} {\bf (iii)} that 
\begin{align*}
I_{\beta,2}'(t)
&\leq 
(1+t)^{-\beta}
\left(\int_{B(0,r_0+t)}|u(x,t)|^p\,dx\right)^\frac{1}{p}
\left(\int_{B(0,r_0+t)}
\psi_{\beta+2,\mu-2}\left(\frac{|x|^2}{(1+t)^2}\right)^{p'}\,dx
\right)^\frac{1}{p'}
\\
&\leq 
(c'_{\beta+2,\mu-2})^{-1}
\|u(t)\|_{L^p}
(1+t)^{-\beta}
\left(
   \int_{B(0,r_0+t)}
   \left(1-\frac{|x|}{1+t}\right)^{(\frac{N+1-\mu}{2}-\beta-1)p'}\,dx\right)^\frac{1}{p'}
\\
&= 
(c'_{\beta+2,\mu-2})^{-1}
|S^{N-1}|^{\frac{1}{p'}}
\|u(t)\|_{L^p}
(1+t)^{\frac{N}{p'}-\beta}
\left(
   \int_0^{\frac{r_0+t}{1+t}}
   (1-\rho)^{-\frac{p'}{q'}}\rho^{N-1}\,d\rho
   \right)^\frac{1}{p'}.
\end{align*}
If $q>p$, then we have 
\begin{align*}
&\ep E_{\beta,0}+\ep E_{\beta,2} t+
\int_0^t
(t-s)G_\beta(s)\,ds
\leq
C_1
\left[
\|u(t)\|_{L^p}
(1+t)^{\frac{N}{p'}+1-\beta}
+
\int_0^t
\|u(s)\|_{L^p}
(1+s)^{\frac{N}{p'}-\beta+\frac{1}{q'}-\frac{1}{p'}}
\,ds
\right].
\end{align*}
By the definition of $\beta$ we have the first desired inequality. 
The second is verified by noticing $q'/p'=1$ in the previous proof. 
\end{proof}

\subsection{Proof of Theorem \ref{main} for subcritical case 
$p_F\leq p<p_0(N+V_0)$}

\begin{lemma}\label{SN}
Assume $p_F(N) \leq p < p_0(N+\mu)$. 
Set 
\[
S_N=\left(0,\frac{1}{p}\right)
\cap 
\left(0,\frac{N-|1-\mu|}{2}\right)
\cap 
\left(\frac{(N-1+\mu)p-(N+1+\mu)}{2}, \frac{(N+1+\mu)p-(N+3+\mu)}{2}\right).
 \]
If $N=1$ and $0<\mu<\frac{4}{3}$, then $S_1\neq \emptyset$.
If $N\geq 2$ and $0\leq \mu<\frac{N^2+N+2}{N+2}$, then 
$S_N\neq \emptyset$. Moreover, one has
\[
\sup S_N=
\begin{cases}
\frac{1}{p}
&
\text{if } N=1\text{ and }\max\{3,\frac{2}{\mu}\}\leq p\leq p_0(1+\mu),
\\[5pt]
\frac{\mu}{2}
&
\text{if } N=1, 0<\mu<\frac{2}{3}\text{ and }3\leq p<\frac{2}{\mu},
\\[5pt]
\frac{1}{p}
&
\text{if } N\geq 2 \text{ and }p_0(N+2+\mu) \leq p<p_0(N+\mu),
\\[5pt]
\frac{(N+1+\mu)p-(N+3+\mu)}{2}
&
\text{if } N\geq 2 \text{ and }p_F(N)\leq p<p_0(N+2+\mu).
\end{cases}
\]
\end{lemma}
\begin{proof}
First we consider the case $N=1$. If $p\geq \frac{2}{1-|1-\mu|}$, then 
$S_1=(0,\frac{1}{p})\cap \left(\frac{\mu p-(2+\mu)}{2}, \frac{(2+\mu)p-(4+\mu)}{2}\right)$. 
Since $\frac{\mu p -(2+\mu)}{2}< \frac{1}{p}$ 
is equivalent to $p< p_0(1+\mu)$ and 
\[
\frac{(2+\mu)p-(4+\mu)}{2}\geq \frac{(2+\mu)p_F(1)-(4+\mu)}{2}=1+\mu>\frac{1}{p},
\]
we have 
$S_1=(\max\{0,\frac{\mu p-(2+\mu)}{2}\},\frac{1}{p})\neq \emptyset$. 
If $p<\frac{2}{1-|1-\mu|}$, then $0<\mu<\frac{2}{3}$ 
and hence
$p<\frac{2}{\mu}$ and 
$S_1=(0,\frac{\mu}{2})\cap \left(\frac{\mu p-(2+\mu)}{2}, 
\frac{(2+\mu)p-(4+\mu)}{2}\right)$. 
Therefore noting that 
\[
\frac{\mu p -(2+\mu)}{2}
<
\frac{-\mu}{2}
<0<\frac{\mu}{2}<\frac{(2+\mu)p-(4+\mu)}{2},
\]
we have $S_1=(0,\frac{\mu}{2})\neq \emptyset$. 

For $N\geq 2$, noting that $|1-\mu|<\frac{N^2}{N+2}$, we have 
$\frac{N-|1-\mu|}{2}>\frac{1}{p_F}\geq \frac{1}{p}$ and then 
$S_N=(0,\frac{1}{p})\cap \left(\frac{(N-1+\mu)p-(N+1+\mu)}{2}, \frac{(N+1+\mu)p-(N+3+\mu)}{2}\right)$. 
We see from $p_F(N)\leq p<p_0(N+\mu)$ that 
$\frac{(N-1+\mu) p -(N+1+\mu)}{2}<\frac{1}{p}$ and 
\[
\frac{(N+1+\mu) p -(N+3+\mu)}{2}
%\geq \frac{(N+1+\mu)(N+2) -N(N+3+\mu)}{2N}
\geq \frac{1+\mu}{N}>0.
\]
This implies that $S_N\neq \emptyset$. The equality for $\sup S_N$ 
can be verified by the above calculations. 
\end{proof}

\begin{proof}[Proof of Theorem \ref{main} for $p_F(N)\leq p<p_0(N+\mu)$]
In this case, fix $q$ satisfying $\frac{1}{q}\in S_N$ in Lemma \ref{SN}. 
This condition is equivalent to 
\[
q>p,\quad
\beta=\frac{N+1-\mu}{2}-\frac{1}{q}>\max\{0,1-\mu\}
\quad
\text{and}\quad
\lambda=\frac{\gamma(N+\mu;p)}{2p}-\frac{1}{p}+\frac{1}{q}\in (0,p-1).
\]
Then we see by Lemma \ref{base2} {\bf (i)} that 
\begin{align}\label{eq:4.2.1}
&
\ep E_{\beta,0}+\ep E_{\beta,1} t+
\int_0^t
(t-s)G_\beta(s)\,ds
\leq
C_1'
\left[
G_\beta(t)^\frac{1}{p}
(1+t)^{\frac{N+1-\beta}{p'}}
+
\int_0^t
G_\beta(s)^\frac{1}{p}
(1+s)^{\frac{N+1-\beta}{p'}-\frac{1}{q}-\frac{1}{p'}}
\,ds
\right].
\end{align}
Observe that 
\begin{align*}
\frac{N+1-\beta}{p'}-\frac{1}{q}-\frac{1}{p}
=
\frac{1}{p}\left(
p-1-\lambda
\right)\in \left(0,\frac{1}{p'}\right).
%=
%\frac{1}{p}\left(
%p-1+\frac{(N-1+\mu)p^2-(N+1+\mu)p-2}{2p}+\frac{1}{p}-\frac{1}{q}
%\right)
\end{align*}
Integrating \eqref{eq:4.2.1} over $[0,t]$, we deduce
\begin{align*}
&
\ep E_{\beta,0}t+\ep \frac{E_{\beta,1}}{2} t^2+
\frac{1}{2}\int_0^t
(t-s)^2G_\beta(s)\,ds
\\
&\leq
C_1'
\left[
\int_0^t G_\beta(s)^\frac{1}{p}
  (1+s)^{\frac{N+1-\beta}{p'}}
\,ds
+
\int_0^t
(t-s)G_\beta(s)^\frac{1}{p}
(1+s)^{\frac{N+1-\beta}{p'}-\frac{1}{q}-\frac{1}{p'}}
\,ds
\right]
\\
&\leq
C_1'
\left(
\int_0^t(1+s)G_\beta(s)\,ds
\right)^{\frac{1}{p}}
\left[
\left(
\int_0^t(1+s)^{(\frac{N+1-\beta}{p'}-\frac{1}{p})p'}\,ds
\right)^{\frac{1}{p'}}
+
\left(
\int_0^t(t-s)^{p'}(1+s)^{(\frac{N+1-\beta}{p'}-\frac{1}{q}-\frac{1}{p})p'-1}\,ds
\right)^{\frac{1}{p'}}
\right]
\\
&\leq
C_2
\left(
\int_0^t(1+s)G_\beta(s)\,ds
\right)^{\frac{1}{p}}
\left[
(1+t)^{\frac{N+1-\beta}{p'}-\frac{1}{p}+\frac{1}{p'}}
+
(1+t)^{1+\frac{N+1-\beta}{p'}-\frac{1}{q}-\frac{1}{p}}
\right]
\\
&\leq
2C_2
\left(
\int_0^t(1+s)G_\beta(s)\,ds
\right)^{\frac{1}{p}}
(1+t)^{1+\frac{p-1-\lambda}{p}}.
%(1+t)^{1+\frac{N-\beta}{p'}-\frac{1}{q}-\frac{1}{p}}.
\end{align*}
We see from the definition of $H_\beta$ that
\begin{align*}
(2C_2)^{-p}(1+t)^{1+\lambda-2p}\left(
\ep E_{\beta,0}t+\ep \frac{E_{\beta,1}}{2} t^2
\right)^p
&\leq H_\beta'(t).
\end{align*}
Hence we have 
\[
H_\beta'(t)\geq 
C_4\ep^{p}(1+t)^{1+\lambda}, \quad t\geq 1. 
\] 
Integrating it over $[0,t]$, we have for $t\geq 2$, 
\[
H_\beta(t)
\geq 
\int_0^t H_\beta'(s)\,ds
\geq 
\int_1^t H_\beta'(s)\,ds
\geq 
C_4\ep^{p}\int_1^t (1+s)^{1+\lambda}\,ds
\geq 
\frac{C_4\ep^p}{2(2+\lambda)}(1+t)^{2+\lambda}. 
\]
We see from the definition of $I_\beta$ that for $t\geq 2$,
\[
J_\beta'(t)=(1+s)^{-3}H_\beta(t)\geq\frac{C_4\ep^p}{2(2+\lambda)}(1+t)^{-1+\lambda}.
\] 
and for $t\geq 4$,
\[
J_\beta(t)=
\int_2^t J_\beta'(s)\,ds
\geq 
\frac{(5^\lambda-3^\lambda)C_4\ep^p}{2\cdot 5^\lambda\lambda(2+\lambda)}
(1+t)^{\lambda}.
\]
On the other hand, 
we see from Lemma \ref{lem:trick} that
\[
(2C_2)^{-p}[J_\beta(t)]^p\leq 
(1+t)^{2-\lambda}J_\beta''(t)+3(1+t)^{1-\lambda} J_\beta'(t)]. 
\]
Moreover, setting 
$J_\beta(t)=\widetilde{J}_\beta(\sigma)$, $\sigma=\frac{2}{\lambda}(1+t)^{\frac{\lambda}{2}}$, 
we have
\[
(1+t)^{1-\frac{\lambda}{2}}J_\beta'(t)=\widetilde{J}_\beta'(\sigma), 
\quad 
(1+t)^{2-\lambda}J_\beta''(t)
+\frac{2-\lambda}{2}(1+t)^{1+\lambda}J_\beta'(t)
=
\widetilde{J}_\beta''(\sigma).
\]
Then 
\begin{gather*}
C_5^{-p}[\widetilde{J}_\beta(\sigma)]^p
\leq 
\widetilde{J}_\beta''(\sigma)
+
\frac{4+\lambda}{\lambda}\sigma^{-1} 
\widetilde{J}_\beta'(\sigma), 
\quad 
\sigma\geq \sigma_0
=
\frac{2}{\lambda}, 
\\
\widetilde{J}_\beta'(\sigma)
\geq 
C_6\ep^p\sigma, 
\quad
\sigma\geq \sigma_1=\frac{2}{\lambda}3^{\frac{\lambda}{2}}, 
\\
\widetilde{J}_\beta(\sigma)
\geq 
C_6\ep^p\sigma^2, 
\quad
\sigma\geq \sigma_2=\frac{2}{\lambda}5^{\frac{\lambda}{2}}.
\end{gather*}
Consequently, by Lemma \ref{lem:blowup} {\bf (i)} we deduce that 
$\widetilde{J}_\beta$ blows up before $C_7\ep^{-\frac{p-1}{2}}$ and then, 
$J_\beta$ blows up before $C_7\ep^{-\frac{p-1}{\lambda}}$. 
By virtue of Lemma \ref{lem:sufficient}, we have 
$\lifespan(u_\ep)\leq C_7\ep^{-\frac{p-1}{\lambda}}$.
Noting that
$1/q$ can be taken as $1/q=\sup S_N-\delta'$ with sufficiently small $\delta'>0$, 
we obtain that 
\[
\lifespan(u_\ep)\leq C_7\ep^{-\theta-\delta}, 
\]
where $\theta$ is given by
\[
\theta=(p-1)\left(\frac{\gamma(N+\mu;p)}{2p}-\frac{1}{p}+\sup S_N\right)^{-1}.
\]
Adding the characterization of $\sup S_N$, 
we complete the proof for subcritical case.
\end{proof}

\subsection{Proof of Theorem \ref{main} for critical case 
$p=p_0(N+\mu)$}
\begin{proof}
In this case we see from $\mu<\mu_*=\frac{N^2+N+2}{N+2}$ that 
$p<\frac{2}{N-|1-\mu|}$. 
Set
\[
\beta_\delta=\frac{N+1-\mu}{2}-\frac{1}{p+\delta}
>\max\{0,1-\mu\}
\]
for $\delta\geq 0$. 
Then by Lemma \ref{base2} {\bf (i)} with $\beta=\beta_\delta$, 
\begin{align*}
\ep E_{\beta_\delta,0}+\ep E_{\beta_\delta,1} t
&\leq 
C_1
\left[
\|u(t)\|_{L^p}
(2+t)^{\frac{N}{p'}+1-\beta_\delta}
+
\int_0^t
\|u(s)\|_{L^p}
(2+s)^{\frac{N}{p'}-\beta_0}
\,ds
\right]
\\
&\leq 
K_1
\left[
\left(G_{\beta_0}(t)\right)^{\frac{1}{p}}
(1+t)^{\frac{N+1-\beta_0}{p'}+(\beta_0-\beta_{\delta})}
+
\int_0^t
\left(G_{\beta_0}(t)\right)^{\frac{1}{p}}
(1+s)^{\frac{N+1-\beta_0}{p'}-1}
\,ds
\right].
\end{align*}
Noting that $\frac{N+1-\beta_0}{p'}=1+\frac{1}{p}$ and  
integrating it over $[0,t]$, we have 
\begin{align*}
\ep E_{\beta_\delta,0}t+\ep \frac{E_{\beta_\delta,1}}{2} t^2
&\leq 
K_1
\left[
\int_0^t \left(G_{\beta_0}(s)\right)^{\frac{1}{p}}
(1+s)^{1+\frac{1}{p}+(\beta_0-\beta_\delta)}\,ds
+
\int_0^t
(t-s)\left(G_{\beta_0}(t)\right)^{\frac{1}{p}}
(1+s)^{\frac{1}{p}}
\,ds
\right]
\\
&\leq 
K_1
\left(
\int_0^tG_{\beta_0}(s)(1+s)\,ds
\right)^{\frac{1}{p}}
\left[
\left(
\int_0^t 
(1+s)^{p'+(\beta_0-\beta_\delta)p'}\,ds
\right)^{\frac{1}{p'}}
+
\left(\int_0^t
(t-s)^{p'}
\,ds
\right)^{\frac{1}{p'}}
\right]
\\
&\leq 
K_2
\left(
\int_0^tG_{\beta_0}(s)(1+s)\,ds
\right)^{\frac{1}{p}}
(1+t)^{1+\frac{1}{p'}}.
\end{align*}
By the definition of $H_{\beta_0}$, we have
for $t\geq 1$, 
\[
H'_{\beta_0}(t)
\geq 
K_2^{-p}
\ep^p\left(E_{\beta_{\delta},0}t+\frac{E_{\beta_\delta,1}}{2} t^2\right)^p
(1+t)^{1-2p}
\geq K_3\ep^p (1+t)
\]
and then for $t\geq 2$, 
\[
H_{\beta_0}(t)\geq \int_1^t\tilde{G}'_{\beta_0}(s)\,ds
\geq 
K_4\ep^p (2+t)^2.
\]
On the other hand, by Lemma \ref{base2} {\bf (ii)} we have 
\begin{align*}
&
\int_0^t
(t-s)G_{\beta_0}(s)\,ds
\leq
C_1
\left[
\|u(t)\|_{L^p}
(1+t)^{\frac{N}{p'}+1-\beta_0}
+
\int_0^t
\|u(s)\|_{L^p}
(1+s)^{\frac{N}{p'}-\beta_0}(\log (1+s))^{\frac{1}{p'}}
\,ds
\right].
\end{align*}
Noting $\frac{N+1-\beta_0}{p'}=1+\frac{1}{p}$ again 
and integrating it over $[0,t]$, we have 
\begin{align*}
&
\frac{1}{2}\int_0^t
(t-s)^2G_{\beta_0}(s)\,ds
\\
&\leq
K_1'
\left[
\int_0^t
G_{\beta_0}(s)^\frac{1}{p}
(1+s)^{\frac{N+1-\beta_0}{p'}}\,ds
+
\int_0^t
(t-s)G_{\beta_0}(s)^\frac{1}{p}
(1+s)^{\frac{N+1-\beta_0}{p'}-1}(\log (2+s))^{\frac{1}{p'}}
\,ds
\right]
\\
&\leq 
K_1'H_{\beta_0}'(t)^{\frac{1}{p}}
\left[
\int_0^t
(1+s)^{p'}\,ds
+
\int_0^t
(t-s)^{p'}
\log (1+s)
\,ds
\right]
\\
&\leq 
K_2'H_{\beta_0}'(t)^{\frac{1}{p}}
(1+t)^{1+\frac{1}{p'}}(\log (2+t))^\frac{1}{p'}.
\end{align*}
As in the proof of subcritial case, 
we deduce
\begin{align*}
(K_2')^{-p}(\log(1+t))^{1-p}J_{\beta_0}(t)^{p}
&\leq H_{\beta_0}'(t)(1+t)^{-1}
\\
&\leq (1+t)^2 J_{\beta_0}''(t)+3(1+t) J_{\beta_0}'(t).
\end{align*}
Here we take $J_{\beta_0}(t)=\widetilde{J}_{\beta_0}(\sigma)$ 
with $\sigma =\log (1+t)$. Since 
\[
(1+t)J_{\beta_0}'(t)=\widetilde{J}_{\beta_0}'(\sigma), 
\quad 
(1+t)^2J_{\beta_0}'(t)+(1+t)J_{\beta_0}''(t)=\widetilde{J}_{\beta_0}''(\sigma),
\]
we obtain for $\sigma\geq \sigma_0:=\log 2$, 
\[
(K_2')^{-p}\sigma^{1-p}\widetilde{J}_{\beta_0}(\sigma)^p
\leq \widetilde{J}_{\beta_0}''(\sigma)+2\widetilde{J}_{\beta_0}'(\sigma).
\]
Moreover, we have 
for $\sigma \geq \sigma_1=\log 4$, 
\begin{align*}
\widetilde{J}_{\beta_0}'(\sigma)
&=(1+t)J_{\beta_0}'(t)
\\
&=(1+t)^{-2}H_{\beta_0}(t)
\\
&\geq K_4\ep^p
\end{align*}
and therefore for $\sigma \geq \sigma_2=2\log 4$, 
\[
\widetilde{J}_{\beta_0}(\sigma)\geq \frac{K_4}{2}\ep^p\sigma. 
\]
Applying Lemma \ref{lem:blowup} {\bf (ii)} we deduce that 
$\widetilde{J}_{\beta_0}$ blows up before $\sigma=K_5\ep^{-p(p-1)}$ 
when $\ep>0$ is sufficiently small. 
Then by definition, the function $J_{\beta_0}$ blows up before $\exp[K_5\ep^{-p(p-1)}]$. 
Consequently, using Lemma \ref{lem:sufficient}, we obtain 
\[
\lifespan (u_\ep)\leq \exp[K_5\ep^{-p(p-1)}].
\]
The proof is complete. 
\end{proof}

\subsection*{Acknowedgements}
This work is partially supported 
by Grant-in-Aid for Young Scientists Research (B) 
No.16K17619 
and 
by Grant-in-Aid for Young Scientists Research (B) 
No.15K17571.

%%%%%%%%%%%%%%                   %%%%%%%%%%%%%%
%%%%%%%%%%%                         %%%%%%%%%%%
%%%%%%%%           References          %%%%%%%%
%%%%%%%%%%%                         %%%%%%%%%%%
%%%%%%%%%%%%%%                   %%%%%%%%%%%%%%
%A B C D E F G H I J K L M N O P Q R S T U V W X Y Z


\begin{thebibliography}{30}

\bibitem{BW}
    R. Beals, R. Wong, 
    ``Special functions,''
    A graduate text. Cambridge Studies in Advanced Mathematics {\bf 126}, 
    Cambridge University Press, Cambridge, 2010.

\bibitem{CHbook} 
    T. Cazenave, A. Haraux, 
    ``An introduction to semilinear evolution equations,'' 
    Translated from the 1990 French original by Yvan Martel and revised by the authors. 
    Oxford Lecture Series in Mathematics and its Applications {\bf 13}. 
    The Clarendon Press, Oxford University Press, New York, 1998.
 
\bibitem{DAbbicco15} 
    M. D'Abbicco, 
    {\it The threshold of effective damping for semilinear wave equations}, 
    Math.\ Methods Appl.\ Sci.\ {\bf 38} (2015), 1032--1045.

\bibitem{DLR15}
    M. D'Abbicco, S. Lucente, M. Reissig, 
    {\it A shift in the Strauss exponent for semilinear wave equations with a not effective damping}, 
    J.\ Differential Equations {\bf 259} (2015), 5040--5073.

\bibitem{IO16}
    M. Ikeda, T. Ogawa, 
    {\it Lifespan of solutions to the damped wave equation with a critical nonlinearity},
    J.\ Differential Equations {\bf 261} (2016), 1880--1903.

\bibitem{IkedaSobajima1}
    M. Ikeda, M. Sobajima, 
    {\it Life-span of blowup solutions to semilinear wave equation with space-dependent critical damping}, preprint.  

\bibitem{John79}
    F. John, 
    {\it Blow-up of solutions of nonlinear wave equations in three space    
    dimensions}, Manuscripta Math.\ {\bf 28} (1979), 235--268.

\bibitem{Kato80}
    T. Kato, 
    {\it Blow-up of solutions of some nonlinear hyperbolic equations}, 
    Comm.\ Pure Appl.\ Math.\ {\bf 33} (1980), 501--505.

\bibitem{LTWarxiv}
    N.A. Lai, H. Takamura, K. Wakasa, 
    {\it Blow-up for semilinear wave equations with the scale invariant 
    damping and super-Fujita exponent}, J. Differential Equations, http://dx.doi.org/10.1016/j.jde.2017.06.017.

\bibitem{Strauss81}
    W.A. Strauss, 
    {\it Nonlinear scattering theory at low energy}, 
    J.\ Funct.\ Anal.\ {\bf 41} (1981), 110--133.

\bibitem{TW11}
    H. Takamura, K. Wakasa, 
    {\it The sharp upper bound of the lifespan of solutions to critical semilinear wave equations in high dimensions}, 
    J.\ Differential Equations {\bf 251} (2011), 1157--1171.

\bibitem{Wakasa16}
    K. Wakasa, 
    {\it The lifespan of solutions to semilinear damped wave equations in one space dimension}, 
    Commun.\ Pure Appl.\ Anal.\ {\bf 15} (2016), 1265--1283.

\bibitem{Wakasugi14}
    Y. Wakasugi, 
    {\it Critical exponent for the semilinear wave equation with scale invariant damping}, 
    Fourier analysis, 375--390, Trends Math., Birkhauser/Springer, Cham, 2014.

\bibitem{Wakasugi-PhD}
    Y. Wakasugi, 
    {\it On the Diffusive Structure for the Damped Wave Equation with Variable Coefficients}, Doctoral thesis, Osaka University, 2014.

\bibitem{YZ06}
    B. Yordanov, Q.S. Zhang, 
    {\it Finite time blow up for critical wave equations in high dimensions}, 
    {J.\ Funct.\ Anal.} {\bf 231} (2006), 361--374.

\bibitem{Zhou07}
    Y. Zhou, 
    {\it Blow up of solutions to semilinear wave equations 
    with critical exponent in high dimensions}, 
    Chin.\ Ann.\ Math.\ Ser.\ B {\bf 28} (2007), 205--212.

\bibitem{ZH14}
    Y. Zhou, W. Han,
    {\it Life-span of solutions to critical semilinear wave equations}, 
    Comm.\ Partial Differential Equations {\bf 39} (2014), 439--451. 

\end{thebibliography}
\end{document}